\documentclass[a4paper, 11pt, reqno]{amsart}
\usepackage{amsmath,amsfonts,amssymb,amsthm,enumerate}
\usepackage{graphicx}
\usepackage{xy} \xyoption{all}

\newtheorem{thm}{Theorem}[section]
\newtheorem{cor}[thm]{Corollary}
\newtheorem{prop}[thm]{Proposition}
\newtheorem{lem}[thm]{Lemma}

\theoremstyle{definition}
\newtheorem{defn}[thm]{Definition}
\newtheorem{exas}[thm]{Example}

\let\phi\varphi

\begin{document}
\title{Purely infinite simple ultragraph Leavitt path algebras}
\maketitle
\begin{center}
T.\,G.~Nam\footnote{Institute of Mathematics, VAST, 18 Hoang Quoc Viet, Cau Giay, Hanoi, Vietnam. E-mail address: \texttt{tgnam@math.ac.vn}} and N. D.~Nam\footnote{Faculty of Pedagogy, HaTinh University, Hatinh, Vietnam. E-mail address: \texttt{nam.nguyendinh@htu.edu.vn}  
		
	\ \ {\bf Acknowledgements}:   The first author was partially supported by the Vietnam Academy of Science and
	Technology grant CT0000.02/20-21.
}
 %The author  expresses their deep gratitude to Professor Gene Abrams (Department of Mathematics, University of Colorado, Colorado Springs, Coloralo, USA) for his valuable suggestions which led to the final shape of the paper.} 
	\end{center}
	
\begin{abstract} In this article, we give necessary and sufficient conditions under which the Leavitt path algebra $L_K(\mathcal{G})$ of an ultragraph $\mathcal{G}$ over a field $K$ is purely infinite simple and that it is von Neumann regular. Consequently, we obtain that every graded simple ultragraph Leavitt path algebra is either a locally matricial algebra, or  a full matrix ring over $K[x, x^{-1}]$, or a purely infinite simple algebra. %(the Trichotomy Principle for graded simple ultragraph Leavitt path algebras).
		\medskip

		
\end{abstract}
	
\section{Introduction}
The study of algebras associated to combinatorial objects has attracted a great deal of attention in the past years. Part of the interest in these algebras arise from the fact that many properties of the combinatorial object translate
into algebraic properties of the associated algebras and their applications to symbolic dynamics. There have been interesting examples of algebras associated to combinatorial objects among which we mention, for example, the following ones: graph $C^*$-algebras, Leavitt path algebras, higher rank graph algebras, Kumjian-Pask algebras, ultragraph $C^*$-algebras (we refer the reader to \cite{a:lpatfd} and \cite{AAS} for a more comprehensive list).

Ultragraphs were defined by Mark Tomforde in \cite{tomf:auatelaacaastg03} as an unifying approach to Exel-Laca and graph $C^*$-algebras. They have proved to be a key ingredient in the study of Morita equivalence of Exel-Laca and graph $C^*$-algebras \cite{kmst:gaelaauacutme}. Recently, Gon\c{c}alvas and Royer have established nice connections between ultragraph $C^*$-algebras and the symbolic dynamics of shift spaces over infinite alphabets (see \cite{gr:uassoia} and \cite{gr:iaessvuatca}).

The Leavitt path algebra associated to an ultragraph was defined by Imanfar, Pourabbas and Larki in \cite{ima:tlpaou}, along with a study of graded ideal structures and a proof of a Cuntz-Krieger uniqueness type theorem. Furthermore, it was shown in \cite{ima:tlpaou} that these algebras provide examples of algebras that can not be realized as the Leavitt path algebra of a graph; that is, the class of ultragraph path algebras is strictly larger than the class of Leavitt path algebras of graphs.
This raises the question of which results about Leavitt path
algebras of graphs can be generalized to ultragraph Leavitt path algebras, and whether results from the $C^*$-algebraic setting can be proved in the algebraic level. Recently, a
number of interesting results regarding this question have been obtained. We mention the following. Gon\c{c}alvas and Royer \cite{gr:iaproulpa} extended to Chen's construction to ultragraph Leavitt path algebras (see \cite{c:irolpa}) of irreducible representations of graph Leavitt path algebras; and Gon\c{c}alvas and Royer \cite{gr:saccfulpavpsgrt} realized ultragraph  Leavitt path algebras as partial skew group rings. Using this realization they characterized artinian ultragraph Leavitt path algebras and gave simplicity criteria for these algebras. The current article
is a continuation of this direction. We give von Neumann reularity and purely infinite simplicity criteria for ultragraph Leavitt path algebras. Consequently, we provide the Trichotomy Principle for graded simple ultragraph Leavitt path algebras.

The article is organized as follows. In section 2, for the reader's convenience, we provide subsequently necessary notions and facts on ultragraphs and ultragraph Leavitt path algebras. We show that the ultragraph Leavitt path algebras arising from acyclic ultragraphs are precisely the von Neumann regular ultragraph Leavitt path algebras, and in this case they are exactly locally matricial algebras (Theorem~\ref{regularity}). In section 3, we give necessary and sufficient conditions on an ultragraph $\mathcal{G}$ so that  $L_K(\mathcal{G})$ is purely infinite simple (Theorem~\ref{purelysimple}). Using Theorems~\ref{regularity} and \ref{purelysimple} and \cite[Theorem~3.4]{ima:tlpaou}, we obtain Theorem~\ref{TrichotomyPrinciple}, showing that every graded simple ultragraph Leavitt path algebra is either a locally matricial algebra, or  a full matrix ring over $K[x, x^{-1}]$, or a purely infinite simple algebra.

\section{Regularity conditions for ultragraph Leavitt path algebras}
The main goal of this section is to establish the equivalence of
the following conditions for an ultragraph $\mathcal{G}$ and a field $K$ (Theorem~\ref{regularity}): (1) $L_K(\mathcal{G})$ is von Neumann regular. (2) $\mathcal{G}$ is acyclic. (3) $L_K(\mathcal{G})$ is a locally matricial $K$-algebra. 

We begin this section by recalling some notions and notes of ultragraph theory introduced by Tomforde in \cite{tomf:auatelaacaastg03} and \cite{tomf:soua}.

An \textit{ultragraph} $\mathcal{G} = (G^0, \mathcal{G}^1, r, s)$ consists of a countable set of vertices $G^0$, a countable set of edges $\mathcal{G}^1$, and functions $s : \mathcal{G}^1 \longrightarrow G^0$ and $r : \mathcal{G}^1 \longrightarrow \mathcal{P}(G^0)\setminus \left\{ \emptyset\right\}$, where $\mathcal{P}(G^0)$ denotes the set of all subsets of $G^0$.

A vertex $v \in G^0$ is called a \textit{sink} if $s^{-1}(v)=\emptyset$ and $v$ is called an \textit{infinite emitter} if $|s^{-1}(v)|=\infty$. A \textit{singular vertex} is a vertex that is either a sink or an infinite emitter. The set of all singular vertices is denoted by $\text{Sing}(\mathcal{G})$. 
A vertex $v \in G^0$ is called a \textit{regular vertex} if $0<|s^{-1}(v)|<\infty$. 

For an ultragraph $\mathcal{G} = (G^0, \mathcal{G}^1, r, s)$ we let $\mathcal{G}^0$ denote the smallest subset of $\mathcal{P}(G^0)$ that contains $\{v\}$ for all $v\in G^0$, contains $r(e)$ for all $e\in \mathcal{G}^1$, and is closed under finite unions and finite intersections. The following lemma gives us another description of $\mathcal{G}^0$.

\begin{lem}[{\cite[Lemma 2.12]{tomf:auatelaacaastg03}}]\label{g^0}	
If $\mathcal{G} := (G^0, \mathcal{G}^1, r, s)$ is an ultragraph, then
$\mathcal{G}^0 = \{\bigcap_{e\in X_1}r(e)\cup \cdots\cup \bigcap_{e\in X_n}r(e)\cup F\mid X_1, \hdots, X_n \text{ are finite subsets of } \mathcal{G}^1\\ \text{ and } F \text{ is a finite subset of } G^0\}.$ Furthermore, $F$ may be chosen to be disjoint from $\bigcap_{e\in X_1}r(e)\cup \cdots\cup \bigcap_{e\in X_n}r(e)$.
\end{lem}

A \textit{finite path} in an ultragraph $\mathcal{G}$ is either an element of $\mathcal{G}^0$ or a sequence $\alpha_1\alpha_2\cdots\alpha_n$ of edges with $s(\alpha_{i+1})\in r(\alpha_i)$ for all $1\le i\le n-1$ and we say that the path $\alpha$ has \textit{length} $|\alpha| := n$. We consider the elements of $\mathcal{G}^0$ to be paths of length $0$. We denote by $\mathcal{G}^*$ the set of all finite paths in $\mathcal{G}$. The maps $r$ and $s$ extend naturally to $\mathcal{G}^*$. Note that when $A\in \mathcal{G}^0$ we define $s(A) = r(A) = A$. An \textit{infinite path} in $\mathcal{G}$  is a sequence $e_1e_2\cdots e_n\cdots$ of edges in $\mathcal{G}$ such that $s(e_{i+1})\in r(e_i)$ for all $i\geq 1$.

If $\mathcal{G}$ is an ultragraph, then a \textit{cycle} in $\mathcal{G}$ is a path $\alpha=\alpha_1\alpha_2\cdots\alpha_{|\alpha|}\in \mathcal{G}^*$ with $|\alpha|\ge 1$ and $s(\alpha)\in r(\alpha)$. An \textit{exit} for a cycle $\alpha$ is one of the following:
\begin{itemize} 
\item[(1)] an edge $e\in \mathcal{G}^1$ such that there exists an $i$ for which $s(e)\in r(\alpha_i)$ but $e\ne \alpha_{i+1}$.
 
\item[(2)] a sink $w$ such that $w\in r(\alpha_i)$ for some $i$. 
\end{itemize}

In \cite{tomf:auatelaacaastg03} Mark Tomforde  introduced  the $C^*$-algebra of an ultragraph as an unifying approach to Exel-Laca and graph $C^*$-algebras. Imanfar, Pourabbas and  Larki in \cite{ima:tlpaou}, introduced the Leavitt path algebra of an ultragraph, along with a study of ideals and a proof of a
Cuntz-Krieger uniqueness type theorem.

\begin{defn}[{cf. \cite[Theorem 2.11]{tomf:auatelaacaastg03} and \cite[Definition 2.1]{ima:tlpaou}}]\label{utraLevittpathalg}
Let $\mathcal{G}$ be an ultragraph and $K$ a field. The \textit{Leavitt path algebra $L_K(\mathcal{G})$ of $\mathcal{G}$ with coefficients in $K$} is the $K$-algebra generated by the set $\{s_e, s^*_{e}\mid e\in\mathcal{G}^1\}$ $\cup\left\{p_{_A}\mid A\in\mathcal{G}^0\right\}$, satisfying the following relations for all $A,B \in \mathcal{G}^0$ and  $e, f\in\mathcal{G}^1$:
\begin{itemize} 	
\item[(1)] $p_{_\emptyset}=0, p_{_A}p_{_B} = p_{_{A\cap B}}$ and $p_{_{A\cup B}} = p_{_A}+p_{_B}-p_{_{A\cap B}}$;
\item[(2)] $p_{s(e)}s_e = s_e = s_ep_{r(e)}$ and $p_{r(e)}s_{e}^* = s_{e}^* = s_{e}^*p_{s(e)}$;
\item[(3)] $s_{e}^*s_f = \delta_{e, f}p_{r(e)}$;
\item[(4)] $p_v = \sum_{s(e)=v}s_es_{e}^*$ for any regular vertex $v$;
\end{itemize} where $p_v$ denotes $p_{_{\{v\}}}$ and $\delta$ is the Kronecker delta.
\end{defn}

We usually denote $s_A := p_{_A}$ for $A\in\mathcal{G}^0$ and $s_\alpha := s_{e_1}\cdots s_{e_n}$ for $\alpha = e_1\cdots e_n \in \mathcal{G}^*$. It is easy to see that the mappings given by $p_{_A}\longmapsto p_{_A}$ for $A\in \mathcal{G}^0$, and $s_e\longmapsto s^*_e$, $s^*_e\longmapsto s_e$ for $e\in \mathcal{G}^1$, produce an involution on the algebra $L_K(\mathcal{G})$, and for any path $\alpha = \alpha_1\cdots\alpha_n$ there exists $s^*_{\alpha} := s^*_{e_n}\cdots s^*_{e_1}$. Also, $L_K(\mathcal{G})$ has the following \textit{universal property}: if $\mathcal{A}$ is a $K$-algebra generated by a family of elements $\{b_A, c_e, c^*_e\mid A\in \mathcal{G}^0, e\in \mathcal{G}^1\}$ satisfying the relations analogous to (1) - (4)  in Definition~\ref{utraLevittpathalg}, then there always exists a $K$-algebra homomorphism $\varphi: L_K(\mathcal{G})\longrightarrow \mathcal{A}$ given by $\varphi(p_A) = b_A$, $\varphi(s_e) = c_e$ and $\varphi(s^*_e) = c^*_e$.	Furthermore, we denote another useful properties as follows.

\begin{lem}\label{graded}
If $\mathcal{G}$ is an ultragraph and $K$ is a field, then the Leavitt path algebra $L_K(\mathcal{G})$ has the following properties:

$(1)$ $(${\cite[Theorem 2.10]{ima:tlpaou} \textnormal{ and } \cite[Theorem 3.10]{gr:saccfulpavpsgrt}}$)$ All elements of the set $\{p_A, s_e, s^*_e\mid A\in \mathcal{G}^0, e\in \mathcal{G}^1\}$ are nonzero.

$(2)$ $(${\cite[Theorem~2.9]{ima:tlpaou}}$)$ $L_K(\mathcal{G})$ is of the form
\begin{center}
$\textnormal{Span}_K\{s_{\alpha}p_{_A}s_{\beta}^*:\alpha,\beta \in \mathcal{G}^*, A\in \mathcal{G}^0 \textnormal{ and } r(\alpha)\cap A \cap r(\beta)\neq \emptyset\}$.	
\end{center}
Furthermore, $L_K(\mathcal{G})$ is a $\mathbb{Z}$-graded $K$-algebra by the grading 
\begin{center}	
$L_K(\mathcal{G})_n=\textnormal{Span}_K\{s_{\alpha}p_{_A}s_{\beta}^*\mid\alpha,\beta \in \mathcal{G}^*, A\in \mathcal{G}^0 \textnormal{ and } |\alpha| -|\beta|=n\}$\quad $(n\in \mathbb{Z})$.
\end{center}
\end{lem}

In the light of Lemma~\ref{graded}, an element $x\in L_K(\mathcal{G})_n$ is called a \textit{homogeneous element of degree} $n$. Recall that a ring $R$ is said to have \textit{local units} if every finite subset of $R$ is contained in a subring of the form $eRe$ where $e = e^2\in R$. The following lemma shows that every ultragraph Leavitt path algebra is an algebra with local units.

\begin{lem}\label{localunit}
Let $\mathcal{G}$ be an ultragraph and $K$ a field. Then $L_K(\mathcal{G})$ is an algebra with local units $($specifically, the set of local units of $L_K(\mathcal{G})$ is given by $\{p_{_A}\mid A\in \mathcal{G}^0\})$. Moreover, $L_K(\mathcal{G})$ is unital if and only if $G^0 \in \mathcal{G}^0$; in this case the identity element is $1=p_{G^0}$.  
\end{lem}
\begin{proof} Consider a finite subset $\{a_i\}^t_{i=1}$ of
$L_K(\mathcal{G})$ and use Lemma~\ref{graded} (2) to write
\begin{center}
$a_i = \sum_{s=1}^{n_i} k^i_s s_{p^i_s}p_{_{A^i_s}}s^*_{q^i_s}$\quad where $k^i_s\in K\setminus \{0\}$, and $p^i_s, A^i_s, q^i_s\in \mathcal{G}^*$. 	
\end{center}
Let $$A:= \{s(p^i_s), s(q^i_s)\mid |p^i_s|\geq 1, |q^i_s|\geq 1\} \bigcup (\bigcup_{|p^i_s|=0=|q^i_s|}s(p^i_s)\cup s(q^i_s))\subseteq G^0.$$
We then have $A\in \mathcal{G}^0$ and $p_{_A} a_i = a_i = a_ip_{_A}$ for all $i$, and so $L_K(\mathcal{G})$ is an algebra with local units.

The remainder follows from \cite[Lemma 2.12]{ima:tlpaou} and, just for the reader's convenience, we briefly sketch it here. Namely, assume that $L_K(\mathcal{G})$ is unital and write 
\begin{center}
$1_{L_K(\mathcal{G})} = \sum_{i=1}^{n} k_i s_{p_i}p_{_{A_i}}s^*_{q_i}$\quad where $k_i\in K\setminus \{0\}$, and $p_i, A_i, q_i\in \mathcal{G}^*$.	
\end{center}
Let $$B:= \{s(p_i)\mid |p_i|\geq 1\} \bigcup (\bigcup_{|p_i|=0}s(p_i))\subseteq G^0.$$ It is obvious that $B\in  \mathcal{G}^0$. If $G^0\notin \mathcal{G}^0$, then there exists an element $v\in G^0\setminus B$, and $$p_v = p_v\cdot 1_{L_K(\mathcal{G})} = p_v(\sum_{i=1}^{n} k_i s_{p_i}p_{_{A_i}}s^*_{q_i}) = 0,$$ a contradiction, which shows that $G^0\in \mathcal{G}^0$. The converse is obvious, thus finishing the proof.
\end{proof}  

\begin{lem}\label{generators}
Let $\mathcal{G}$ be an ultragraph and $K$ a field. Then the algebra $L_K(\mathcal{G})$ is generated by $\left\{s_e, s_e^*\mid e\in  \mathcal{G}^1\right\} \cup \left\{p_v\mid v \in \text{Sing}(\mathcal{G})\right\}$. 
\end{lem}

\begin{proof} 
Let $\mathcal{A}$ be the $K$-subalgebra of $L_K(\mathcal{G})$ generated by  $\{ s_e, s_e^*\mid e\in  \mathcal{G}^1\} \cup \{p_v\mid v \in \text{Sing}(\mathcal{G})\}$. We claim that $L_K(\mathcal{G})\subseteq \mathcal{A}$. To do so, it is sufficient to show that $p_A\in \mathcal{A}$ for all $A\in \mathcal{G}^0$. Take any $A\in \mathcal{G}^0$. By Lemma~\ref{g^0}, there exist finite subsets  $X_1, \cdots, X_n$ of $\mathcal{G}^1$ and finite subset $F$ of $\mathcal{G}^0$ such that
\begin{center}
$A = \bigcap_{e\in X_1}r(e)\cup \cdots\cup \bigcap_{e\in X_n}r(e)\cup F \text{ and } F\cap (\bigcap_{e\in X_1}r(e)\cup \cdots\cup \bigcap_{e\in X_n}r(e)) = \varnothing$.
\end{center}

Note that $p_v\in \mathcal{A}$ for every singular vertex $v$. Also, if $v$ is a regular vertex, then $p_v=\sum_{e\in s^{-1}(v)}s_es_e^* \in \mathcal{A}$. This implies that $p_{_F} = \sum_{v\in F}p_v\in \mathcal{A}$.

For every $e \in \mathcal{G}^1$, we always have that $p_{r(e)}=s_e^*s_e \in \mathcal{A}$, and so $$p_{_{A_i}} = \prod_{e\in X_i}p_{r(e)}\in \mathcal{A}$$ for all $1\leq i\leq n$, where $A_i := \bigcap_{e\in X_i}r(e)$. By induction on $n$ we obtain that $$p_{\bigcup\limits_{i=1}^nA_i} =\sum\limits_{k=1}^n(-1)^{k+1}\sum\limits_{1\le i_1<i_2<...<i_k\le n}p_{_{A_{i_1}\cap A_{i_2}\cap...\cap A_{i_n}}}\in \mathcal{A}.$$ This implies that $p_{_A} = p_{\bigcup\limits_{i=1}^nA_i} + p_{_F}\in \mathcal{A}$, and hence $L_K(\mathcal{G})\subseteq \mathcal{A}$. The inverse inclusion is obvious, and so $L_K(\mathcal{G}) = \mathcal{A}$, thus finishing our proof. 	
\end{proof}

Let $\mathcal{G}= (G^0, \mathcal{G}^1, r, s)$ be an ultragraph and let $F$ be a finite subset of $\mathcal{G}^1 \cup \text{Sing}(\mathcal{G})$ (where we denote by $\text{Sing}(\mathcal{G})$ the set of all singular vertecies of $\mathcal{G}$). Write $F^0 := F$ $\cap$ $\text{Sing}(\mathcal{G})$ and $F^1 :=F\cap \mathcal{G}^1=\{e_1, e_2,\hdots,e_n\}$. Following \cite{ima:tlpaou}, we construct a finite graph $G_F$ as follows. For each $\omega=(\omega_1,\hdots,\omega_n)\in \{0,1\}^n \setminus \left\{0^n\right\}$, we define
\begin{center}
$r(\omega):=\bigcap_{\omega_i=1}r(e_i)\setminus \bigcup_{\omega_j=0}r(e_j)$ and $R(\omega):=r(\omega)\setminus F^0$. 	
\end{center} 
Notice that $r(\omega)\cap r(\nu) = \varnothing$ for distinct $\omega, \nu\in \{0,1\}^n \setminus \{0^n\}$. Let 
\begin{center}
$\Gamma_0:= \{\omega\in \{0,1\}^n \setminus \{0^n\}\mid$ there are vertices $v_1,\hdots, v_m$ such that  	
\end{center}
\hspace{3 cm}$R(w)=\{v_1, \cdots, v_m\}$ and $\emptyset \ne s^{-1}(v_i) \subseteq F^1$ for $1\le i \le m\}$	
and 
$$\Gamma_F:= \{\omega\in \{0,1\}^n \setminus \{0^n\}\mid R(w)\ne \varnothing \text{ and } \omega\notin \Gamma_0\}.$$
	
Now we define the finite graph $G_F=(G_F^0, G_F^1, r_F, s_F)$ as follows:
\begin{equation*}
\begin{array}{rcl}
G_F^0&:=&F^0 \cup F^1 \cup \Gamma_F, \text{ and }\\ \\
G_F^1&:=& \{(e,f)\in F^1\times F^1\mid s(f)\in r(e)\}\\
&& \cup\, \{(e,v)\in F^1\times F^0\mid v \in r(e)\}\\
&&\cup\,  \{(e,\omega)\in F^1\times\Gamma_F\mid\omega_i=1 \text{ when } e=e_i\}
\end{array}
\end{equation*}
with
\begin{equation*}
\begin{array}{rcl}
s_F((e,f))=e&\hspace{2 cm} &s_F((e,v))=e \hspace{2,5 cm} s_F((e,\omega))=e\\ 
r_F((e,f))=f& \hspace{2 cm}&r_F((e,v))=v \hspace{2,5 cm} r_F((e,\omega))=\omega.
\end{array}
\end{equation*}

As usual, an ultragraph is called \textit{acyclic} if it has no cycles. The following lemma gives us a criterion for acyclic ultragraphs.

\begin{lem}\label{acyclic}
An ultragraph $\mathcal{G}$ is acyclic if and only if $G_F$ is acyclic for every non-empty finite subset $F$ of $\mathcal{G}^1\cup \rm{Sing}(\mathcal{G})$.
\end{lem}

\begin{proof}
($\Longrightarrow$) Assume that $\mathcal{G}$ is an acyclic ultragraph and $F$ is a finite subset of $\mathcal{G}^1\cup \rm{Sing}(\mathcal{G})$. We claim that the graph $G_F$ is acyclic. Indeed, suppose $G_F$ is not acyclic, that means, it
has a cycle $c=c_1c_2\cdots c_m$, where $c_i\in G_F^1$ for all $i$. By the definition of $G_F$, we must have that $c_i \in \{(e,f)\in F^1\times F^1: s(f)\in r(e)\}$ for all $i$. So, without loss of generality, we may assume $c_i = (e_i, f_i)\in F^1\times F^1$ with $s(f_i)\in r(e_i)$. Then, since $c= c_1c_2\cdots c_m$ is a cycle in $G_F$, we have that $e_1f_1e_2f_2\hdots e_mf_m$ is a cycle in $\mathcal{G}$, a contradiction, and hence $G_F$ is acyclic.
	
($\Longleftarrow$) Assume that $\mathcal{G}$ has a cycle $\alpha=e_1e_2\cdots e_n$, where $e_i\in \mathcal{G}^1$ for all $i$. Let $F:= \{e_1, e_2,\hdots, e_n\} \subset \mathcal{G}^1$. Since $\alpha$ is a cycle in $\mathcal{G}$, $s(e_{i+1})\in r(e_i)$ for all $1\leq i\leq n-1$, and $s(e_1)\in r(e_n)$. This implies that $(e_i,e_{i+1})\in G_F^1$ for all $1\leq i\leq n-1$  and $(e_n,e_1)\in G_F^1$, and $(e_1,e_2) (e_2,e_3)\cdots (e_n,e_1)$ is a cycle in the garph $G_F$, thus finishing our proof.
\end{proof}

A (not necessarily unital) ring $R$ is called \textit{von Neumann regular} in case for every $r\in R$ there exists $s\in R$ such that $r = rsr$. A \textit{matricial $K$-algebra} is a finite direct sum of full finite dimensional matrix algebras over the field $K$. A \textit{locally matricial $K$-algebra} is a direct limit of matricial $K$-algebras (with not necessarily-unital transition homomorphisms). In \cite[Theorem 1]{ar:rcfalpa} Abrams and Rangaswamy showed that the graph Leavitt path algebras arising from acyclic graphs are precisely the von Neumann regular Leavitt path algebras, and in this case they are exactly locally matricial algebras. The following theorem extends this result to ultragraph Leavitt path algebras.

\begin{thm}\label{regularity}
Let $\mathcal{G}$ be an ultragraph and $K$ a field. Then the following conditions are equivalent:

$(1)$ $L_K(\mathcal{G})$ is von Neumann regular;

$(2)$ $\mathcal{G}$	is acyclic;

$(3)$ $L_K(\mathcal{G})$ is a locally matricial $K$-algebra, that is, $L_K(\mathcal{G})$ is a union of a chain of matricial $K$-subalgebras.
\end{thm} 
\begin{proof}
(1)$\Longrightarrow $(2). The proof is essentially based on the ideas in the proof of (2)$\Longrightarrow$(3) in \cite[Theorem 1]{ar:rcfalpa}, by using Lemma~\ref{graded}.

Assume that $L_K(\mathcal{G})$ is von Neumann regular, and there exists a cycle $c$ in $\mathcal{G}$; denote $s(c)$ by $v$. Since $L_K(\mathcal{G})$ is von Neumann regular, there exists an element $\beta\in L_K(\mathcal{G})$ such that $p_v - s_cp_v = (p_v - s_cp_v)\beta (p_v - s_cp_v)$. Replacing $\beta$ by $p_v\beta p_v$ if necessary, there is no loss of generality in assuming that $\beta = p_v\beta p_v$. By Lemma~\ref{graded}, we may write $\beta$ as a sum of homogeneous elements $\beta = \sum ^n_{i=m}\beta_i$, where $m, n\in \mathbb{Z}$, $\beta_m \neq 0$, $\beta_n \neq 0$, and $\rm{deg}$$(\beta_i) =i$ for all nonzero $\beta_i$ with $m \leq i\leq n$. Since $\rm{deg}$$(p_v) = 0$, we have $p_v\beta_i p_v = \beta_i$ for all $i$. Then \[p_v - s_cp_v = (p_v - s_cp_v)(\sum ^n_{i=m}\beta_i) (p_v - s_cp_v).\] Equating the lowest degree terms on both sides, we obtain that $\beta_m = p_v$. Since $\rm{deg}$$(p_v) = 0$, we must have that $m=0$ and $\beta_0 = p_v$. Thus $\beta = \sum^n_{i=0}\beta_i$. Let $\rm{deg}$$(s_c) = s >0$. By again equating terms of like degree in the displayed equation, we see that $\beta_i = 0$ whenever $i$ is nonzero and not a multiple of $s$, so that \[\sum ^n_{i=m}\beta_i = p_v + \sum_{t=1}^{k}\beta_{ts}.\] We then have $$p_v -s_cp_v = (p_v - s_cp_v)p_v(p_v - s_cp_v) + (p_v - s_cp_v)(\sum_{t=1}^{k}\beta_{ts}) (p_v - s_cp_v),$$ which shows that $0 = -p_v + (s_cp_v)^2 + (p_v - s_cp_v)(\sum_{t=1}^{k}\beta_{ts}) (p_v - s_cp_v)$. By equating the degree $s$ components on both sides we obtain $\beta_s = s_cp_v$. Similarly, by equating the degree $2s$
components, we obtain $0 = (s_cp_v)^2 - (s_cp_v)\beta_s - \beta_s(s_cp_v) + \beta_{2s}$, so $\beta_{2s} = (s_cp_v)^2$, and continuing in this manner we get $\beta_{ts} = (s_cp_v)^t$ for all $t$. In particular, we conclude that every homogeneous component $\beta_i$ of $\beta$ commutes with $s_cp_v$ in $L_K(\mathcal{G})$. This yields that $(s_cp_v)\beta = \beta (s_cp_v)$. But then the equation $p_v - s_cp_v = (p_v - s_cp_v)\beta (p_v - s_cp_v)$ becomes
\[p_v - s_cp_v = \beta (p_v - s_cp_v)^2.\]

But this is not possible, as follows. Let $i$ be maximal with the property that $\beta_i ((p_v - s_cp_v)^2)\neq 0$ (Such $i$ exists, since $\beta_0 = p_v$ has this property.) Then the right hand side contains terms of degree $2s +i$ (namely, $\beta_i (s_cp_v)^2$), while the maximum degree of terms
on the left hand side is $s$.

(2)$\Longrightarrow$(3). Since $G^0$ and $\mathcal{G}^1$ are countable sets, we may rename the edges of $\mathcal{G}^1$ as
a sequence $\{e_i\}^{\infty}_{i=1}$ and rename the vertices of $\text{Sing}(\mathcal{G})$ a sequence $\{v_i\}^{\infty}_{i=1}$.
For $n\in \mathbb{N}$, we denote by $B_n$ the subalgebra of $L_K(\mathcal{G})$ generated by $\{s_{e_i}, s^*_{e_i}, p_{v_i}\mid i = 1, \hdots, n\}$. We then have that $B_n \subseteq B_{n+1}$ for all $n$, and by Lemma~\ref{generators}, $L_K(\mathcal{G}) = \bigcup_{n=1}^{\infty}B_n$. By \cite[Lemma 2.13]{ima:tlpaou}, we have that $B_n \cong L_K(G_{F_n})$, where 
$F_n :=\{e_1,\hdots,e_n\}\cup \{v_1,\hdots,v_n\}$. Since $\mathcal{G}$ is acyclic, and by Lemma~\ref{acyclic}, the finite graph $G_{F_n}$ is acyclic for all $n$. By \cite[Proposition 3.5]{ap:fdlpa}, $L_K(G_{F_n})$ is a matricial $K$-algebra. Therefore, $L_K(\mathcal{G})$ is a locally matricial $K$-algebra.

(3)$\Longrightarrow $(1).  It is well-known that every matricial $K$-algebra is a von Neumann regular ring, and hence easily so too is any direct union of such algebras, thus finishing our proof.
\end{proof}

\section{Purely infinite simplicity}

The main goal of this section is both to give a graph-theoretic criterion for purely infinite simple ultragraph Leavitt path algebras (Theorem~\ref{purelysimple}) and provide a complete description of graded simple ultragraph Leavitt path algebras (Theorem~\ref{TrichotomyPrinciple}).

Recall (see e.g. \cite{agp:kopisrr}) that an idempotent $e$ in a ring $R$ is called \textit{infinite} if $eR$ is isomorphic as a right $R$-module to a proper direct summand of itself. $R$ is called \textit{purely infinite} in case every nonzero right ideal of $R$ contains an infinite idempotent. The following lemma provides us with a useful criterion for purely infinite simple rings with local units.

\begin{lem}[{\cite[Proposition 10]{ap:pislpa06}}]\label{pisimple}
For any ring with local units $R$, the following conditions are equivalent:	

$(1)$ $R$ is purely infinite simple;

$(2)$ $R$ is not a division ring, and $R$ has the property that for every pair of nonzero elements $\alpha, \beta\in R$ there exist elements $a, b\in R$ such that $a\alpha b = \beta$. 
\end{lem}

We will use the above lemma to characterize the purely infinite simplicity of ultragraph Leavitt path algebras. Before doing so, we recall some notions and facts introduced by Tomforde in \cite{tomf:soua}. Let $\mathcal{G}$ be an ultragraph. A subset $\mathcal{H}\subseteq\mathcal{G}^0$ is called \textit{hereditary} if the following conditions are satisfied:
\begin{itemize} 
\item[(1)] whenever $e$ is an edge with $\{s(e)\} \in\mathcal{H}$, then $r(e)\in \mathcal{H}$;
	
\item[(2)] $A\cup B \in \mathcal{H}$ for all $A, B \in \mathcal{H}$;

\item[(3)] $A\in \mathcal{H}, B\in \mathcal{G}^0$ and $B\subseteq A$, imply that $B\in\mathcal{H}$.
\end{itemize}	
A subset $\mathcal{H}\subseteq\mathcal{G}^0$ is called \textit{saturated} if for any $v\in G^0$ with $0<|s^{-1}(v)|<\infty$, we have that
\begin{center}
$\left\{r(e)\mid e\in \mathcal{G}^1 \text{ and } s(e)=v\right\}\subseteq \mathcal{H}$ implies $\left\{v\right\}\in\mathcal{H}$.	
\end{center} 
Note that $\varnothing$ and $\mathcal{G}^0$ are two saturated hereditary subsets of $\mathcal{G}^0$. We denote by $\mathcal{H}_{\mathcal{G}}$ the set of all saturated hereditary subsets of $\mathcal{G}^0$. For $\mathcal{H} \subseteq \mathcal{G}^0$, we denote by
$\overline{\mathcal{H}}$ the smallest saturated hereditary subset of $\mathcal{G}^0$ containing $\mathcal{H}$. In \cite[Lemma 3.12]{tomf:soua} Tomforde gave a useful description of $\overline{\mathcal{H}}$ as follows.

\begin{lem}[{\cite[Lemma 3.12]{tomf:soua}}] \label{hereditarysaturationsmallest}
Let $\mathcal{G}:= (G^0, \mathcal{G}^1, r, s)$ be an ultragraph and let $\mathcal{H} \subseteq \mathcal{G}^0$ be a hereditary subset. Set $\mathcal{H}_0:= \mathcal{H}$ and for $n\in \mathbb{N}$ define
\begin{center} 
		$\mathcal{H}_{n+1}:= \{A\cup F\mid A\in \mathcal{H}_n$ and $F \text{ is a finite subset of } S_n\}$
\end{center}
where $S_n:= \{w\in G^0\mid 0<|s^{-1}(w)|<\infty$ and $\{r(e)\mid s(e)=w\}\subseteq \mathcal{H}_n\rbrace$. Then $\overline{\mathcal{H}}=\bigcup\limits_{i=0}^{\infty}\mathcal{H}_i$ and every $X\in \overline{\mathcal{H}}$ has the form $X=A\cup F$ for some $A\in \mathcal{H}$ and
for some finite set $F\subseteq \bigcup_{i=1}^{\infty}S_i$.
\end{lem}

Following \cite{tomf:soua}, if $\mathcal{G}$ is an ultragraph and $v,w\in \mathcal{G}^0$, we write $w\ge v$ to mean that there exits a path $\alpha \in \mathcal{G}^*$ with $s(\alpha)=w$ and $v\in r(\alpha)$. Also, we write $G^0\ge \{v\}$ to mean that $w\ge v$ for all $w\in G^0$. 
We say that a vertex $w$ \textit{connects to a cycle} $\alpha = e_1\cdots e_n$ if there exists a path $\beta\in \mathcal{G}^*$
with $s(\beta) = w$ and $s(e_i) \in r(\beta)$ for some $1 \leq i \leq n$. Note that if $w$ is a sink on a cycle (i.e. $w \in r(e_i)$ for some $i$), then $w$ does not connect to a cycle.
We say that a vertex $w$ \textit{connects to an infinite path} $\alpha = e_1\cdots e_n\cdots$ if there exists an $i\in \mathbb{N}$ such that $v\ge s(e_i)$.

If $v\in G^0$ and $A\subseteq G^0$, then we write $v \rightarrow A$ to mean that there exist a finite number of paths $\alpha_1,\hdots,\alpha_n \in \mathcal{G}^*$ such that $s(\alpha_i)=v$ for all $1\le i \le n$ and $A\subseteq \cup_{i=1}^nr(\alpha_{i})$. Note that if $A=\{w\}$, then $v \rightarrow \{w\}$ if and only if $v\ge w$.

The following lemma provides us with a criterion for ultragraphs $\mathcal{G}$ with $\mathcal{H}_{\mathcal{G}}= \{\varnothing, \mathcal{G}^0\}$. The proof given in the one of \cite[Theorem 3.11]{tomf:soua} applies; and just for the reader's convenience, we reproduce it here.

\begin{lem}\label{gradedsimpleulgraph}
For every ultragraph $\mathcal{G}$, $\mathcal{H}_{\mathcal{G}}= \{\varnothing, \mathcal{G}^0\}$ if and only if the following conditions are satisfied:
	
$(1)$ Every vertex connects to every infinite path.
	
$(2)$ $G^0\ge \{v\}$ for every singular vertex $v\in G^0$.
	
$(3)$ If $e\in \mathcal{G}^1$ is an edge for which the set $r(e)$ is infinite, then for every $w\in G^0$ there exists a set $A_w\subseteq r(e)$ for which $r(e)\setminus A_w$ is finite and $w \rightarrow A_w$.
\end{lem}
\begin{proof}
	
($\Longrightarrow$)	Let $\alpha=e_1e_2\cdots$ be an infinite path and set $K:=\{w\in G^0\mid w\ngeq s(e_i) \text{ for all } i\}$. Let $\mathcal{H}:= \left\{A\in \mathcal{G}^0\mid A\subseteq K\right\}$. We then claim that $\mathcal{H}$ is a saturated hereditary subset of $\mathcal{G}^0$. Indeed, let $e\in \mathcal{G}^1$ with $\{s(e)\}\in \mathcal{H}$ (i.e. $s(e)\in K$). If $r(e) \notin \mathcal{H}$, then $r(e)\nsubseteq K$, and so there exist $w\in r(e)$ and $i\in \mathbb{N}$ such that $w\ge s(e_i)$. This implies that $s(e)\ge w \ge s(e_i)$, that means, $s(e)\notin K$, a contradiction, and hence $r(e)\in \mathcal{H}$.
It is also obvious that $A\cup B\in \mathcal{H}$ for all $A, B\in \mathcal{H}$, and $B\in \mathcal{H}$ for all $B\subseteq A\in \mathcal{H}$. Therefore, $\mathcal{H}$ is a hereditary subset of $\mathcal{G}^0$.
	
Let $v\in G^0$ with $0<|s^{-1}(v)|<\infty$ and $\{r(e)\mid e\in \mathcal{G}^1 \text{ and } s(e)=v\}\subseteq \mathcal{H}$.
Assume that $\{v\} \notin \mathcal{H}$. Then, there exists $i\in \mathbb{N}$ such that $v\ge s(e_i)$, that means, there exists a path $p=f_1f_2\cdots f_m\in \mathcal{G}^*$ such that $s(p)=s(f_1)=v$ and $s(e_i)\in r(p)$. This implies that $s(f_2)\ge s(e_i)$, and so $s(f_2) \notin K$. On the other hand, since $r(f_1)\subseteq K$ and $s(f_2)\in r(f_1)$, we must have that $s(f_2)\in K$, a contradiction, which shows that  $\{v\} \in \mathcal{H}$. So $\mathcal{H}$ is a saturated subset of $\mathcal{G}^0$, thus proving the claim.
	
Since $\{s(e_1)\}\notin \mathcal{H}$, we have that $\mathcal{H} \ne \mathcal{G}^0$, and so $\mathcal{H}=\varnothing$, by our hypothesis that $\mathcal{H}_{\mathcal{G}}= \{\varnothing, \mathcal{G}^0\}$. So, every vertex $v\in G^0$ there exists an $i\in \mathbb{N}$ such that $v\ge s(e_i)$, showing item (1).
	
Let $v\in G^0$ be a singular vertex. Fix any vertex $w\in G^0$ and definite $K:= \left\{x\in G^0\mid w\ge x\right\}$. Let $\mathcal{H}:= \{A\in \mathcal{G}^0\mid A\subseteq K\}$. Similar to the above case, we obtain that $\mathcal{H}$ is a hereditary subset of $\mathcal{G}^0$. 
	
We denote by $\overline{\mathcal{H}}$ the smallest saturated hereditary subset of $\mathcal{G}^0$ containing $\mathcal{H}$. Since $\{w\}\in \mathcal{H}$, $\overline{\mathcal{H}} \ne \emptyset$, and so $\overline{\mathcal{H}}=\mathcal{G}^0$. Using the notation of Lemma~\ref{hereditarysaturationsmallest} and since $v$ is a singular vertex, we get that $v\notin S_i$ for all $i$. By Lemma~\ref{hereditarysaturationsmallest} and since  $\{v\}\in \overline{\mathcal{H}}$, we immediately get that $\left\{v\right\}\in \mathcal{H}$. This implies that $v\in K$ and $w\ge v$. Hence $G^0\ge \left\{v\right\}$, proving item (2).
	
Let $e\in \mathcal{G}^1$ be an edge such that $r(e)$ is an infinite set. Fix $w\in G^0$ and set $\mathcal{H}:= \left\{A\in \mathcal{G}^0\mid w\rightarrow A\right\}$. We claim that $\mathcal{H}$ is a hereditary subset of $\mathcal{G}^0$. Indeed, 
if $f$ is an edge with $\left\{s(f)\right\}\in \mathcal{H}$, then $w \rightarrow \{s(f)\}$, and hence $w\ge s(f)$. Thus, there exists a path $\beta\in \mathcal{G}^*$ such that $s(\beta)=w$ and $s(f)\in r(\beta)$. We then have that $w \rightarrow r(f)$ via the path $\beta f$. Thus, $r(f)\in \mathcal{H}$.
	
If $A, B \in \mathcal{H}$, then $w \rightarrow A$ and $w \rightarrow B$, and so there exist a finite number of paths $\alpha_1,\hdots,\alpha_n \in \mathcal{G}^*$  such that $s(\alpha_i)=w$ for all $1\le i \le n$ and $A\subseteq \bigcup_{i=1}^nr(\alpha_i)$, and  there exist a finite number of paths $\beta_1,\hdots,\beta_m \in \mathcal{G}^*$ such that $s(\beta_j)=w$ for all $1\le j \le m$ and $B\subseteq \bigcup_{j=1}^mr(\beta_j)$. It is obvious that $A\cup B\subseteq (\bigcup_{i=1}^nr(\alpha_i))\cup (\bigcup_{j=1}^mr(\beta_j))$, and so $w\rightarrow A\cup B$. Therefore, $A\cup B \in \mathcal{H}$.
	
If $A\in \mathcal{H}, B\in \mathcal{G}^0$, and $B\subseteq A$, then we have that $w \rightarrow A$, and so there exist a finite number of paths $\alpha_1,\hdots,\alpha_n \in \mathcal{G}^*$  such that $s(\alpha_i)=w$ for all $1\le i \le n$ and $A\subseteq \bigcup_{i=1}^nr(\alpha_i)$. This implies that $B\subseteq \bigcup_{i=1}^nr(\alpha_i)$, and hence $B\in \mathcal{H}$, thus proving the claim.
	
Since $\left\{w\right\} \in \mathcal{H}$, $\mathcal{H}$ is nonempty, and hence $\overline{\mathcal{H}}=\mathcal{G}^0$.  Thus $r(e)\in \overline{\mathcal{H}}$. By Lemma \ref{hereditarysaturationsmallest}, $r(e)=A_w\cup F$ for some $A_w\in \mathcal{H}$ and some finite set $F\subseteq \bigcup_{i=1}^{\infty}S_i$. Then $w\rightarrow A_w$ and $r(e)\setminus A_w$ is finite, showing item (3).
	
($\Longleftarrow$) Let $\mathcal{H}$ be a nonempty saturated hereditary subset of $\mathcal{G}^0$. We first claim that for every $w\in \mathcal{G}^0$ with $\left\{w\right\} \notin \mathcal{H}$, there exists an edge $e\in \mathcal{G}^1$ such that $s(e)=w$ and $r(e)$ contains a vertex $w'$ for which $\left\{w'\right\}\notin \mathcal{H}$. Indeed, since 
$\{w\}\notin \mathcal{H}$ and $G^0\ge \left\{v\right\}$ for every singular vertex $v$, $w$ is a regular vertex. For otherwise, if $w$ is a singular then $G^0\ge \{w\}$, but since $\mathcal{H}\ne \varnothing$ so that there exists a vertex $v\in G^0$ such that $\{v\}\in \mathcal{H}$ and $v\ge w$. Since $\mathcal{H}$ is hereditary, we must have that $\{w\}\in \mathcal{H}$, a contradiction.
	
Now, since $\mathcal{H}$ is saturated, there exists an edge $e\in \mathcal{G}^1$ such that $s(e)=w$ and $r(e)\notin \mathcal{H}$. If $r(e)$ is finite, then since $\mathcal{H}$ is closed under unions, there must exist a vertex $w'\in r(e)$ such that $\{w'\}\notin \mathcal{H}$, as desired. Consider the case that $r(e)$ is infinite. If  $\{v\} \notin \mathcal{H}$ every vertex $v\in r(e)$, then the claim is obvious. If there exists a vertex $x\in G^0$ such that $\{x\}\in \mathcal{H}$, then there exists $A_x\subseteq r(e)$ such that $w\rightarrow A_x$ and $r(e)\setminus A_x$ is a finite set. Let $\alpha_1, \alpha_2,\hdots, \alpha_n\in \mathcal{G}^*$ be paths with $s(\alpha_i)=x$ and $A_x\subseteq \bigcup_{i=1}^nr(\alpha_i)$. Since $\{x\}\in \mathcal{H}$ and $\mathcal{H}$ is hereditary,
$r(\alpha_i)\in \mathcal{H}$ for all $i$, and so
$\bigcup_{i=1}^nr(\alpha_i)\in \mathcal{H}$.  If $r(e)\setminus A_x\in \mathcal{H}$, then $\bigcup_{i=1}^nr(\alpha_i)\cup (r(e)\setminus A_x)$ is an element in $\mathcal{H}$ containing $r(e)$. Then, since $\mathcal{H}$ is hereditary, we immediately get that $r(e)\in \mathcal{H}$, a contradiction. Therefore, we must have that $r(e)\setminus A_x$ is not in $\mathcal{H}$, that means, $\{w'\}\notin \mathcal{H}$ for some $w'\in r(e)\setminus A_x$, as desired. 
	
Now suppose that $\mathcal{H}\neq \mathcal{G}^0$. Then, 
there exists $\left\{w_1\right\}\notin \mathcal{H}$. By the above claim, there exist an edge $e_1$ and a vertex $w_2$ such that $s(e_1)=w_1, w_2 \in r(e_1)$, and $\left\{w_2\right\}\notin\mathcal{H}$. Continuing inductively, we create an infinite path $e_1e_2e_3\cdots$ with $\{s(e_i)\}\notin \mathcal{H}$ for all $i$. Since $\mathcal{H}\neq \varnothing$, there exists a vertex $v\in \mathcal{G}^0$ with $\{v\}\in \mathcal{H}$. By condition (1), we obtain that $v\ge s(e_n)$ for some $n$. Then, since $\mathcal{H}$ is hereditary and $\{v\}\in \mathcal{H}$, we must have that $\{s(e_n)\}\in \mathcal{H}$, a contradiction. It implies that  $\mathcal{H} =\mathcal{G}^0$, thus finishing the proof.
\end{proof}

We now have all the necessary ingredients in hand to prove the first main result of this section, which both characterizes the purely infinite simple Leavitt path algebras in terms of properties of the associated graph and extends Abrams and Aranda-Pino's result \cite[Theorem 11]{ap:pislpa06} to  ultragraph Leavitt path algebras.

\begin{thm}\label{purelysimple}
Let $\mathcal{G}$ be an ultragraph and $K$ a field. Then $L_K(\mathcal{G})$ is purely infinite simple if and only if the following three conditions are satisfied: 

$(1)$ The only hereditary and saturated subsets of $\mathcal{G}^0$ are $\emptyset$ and $\mathcal{G}^0$;

$(2)$ Every cycle in $\mathcal{G}$ has an exit;

$(3)$ Every vertex connects to a cycle.\\
Equivalently, $(3)$ may be replaced by:

$(3')$ $\mathcal{G}$ contains at least one cycle.     
\end{thm}

\begin{proof} ($\Longrightarrow$) Assume that $L_K(\mathcal{G})$ is purely infinite simple. By \cite[Theorem 4.7]{gr:saccfulpavpsgrt} we have (1) and (2). If $\mathcal{G}$
has no cycles, i.e. $\mathcal{G}$ is acyclic, then by Theorem~\ref{regularity}, $L_K(\mathcal{G})$ is a locally matricial $K$-algebra, that is, it is a direct union of matricial $K$-subalgebras. Since every matricial $K$-algebra is finite dimentional, and by \cite[Lemma 8]{ap:pislpa06}, $L_K(\mathcal{G})$ is not purely infinite, a contradiction. Hence, $\mathcal{G}$ must contain at least one cycle $c$. We then have an infinite path $c^{\infty}:= cc\cdots c\cdots$ in $\mathcal{G}$. By Lemma~\ref{gradedsimpleulgraph}, every vertex connects to $c^{\infty}$; equivalently, every vertex connects to $c$, proving (3). Notice that the above argument shows that conditions (3) and (3') are equivalent in the presence of conditions (1) and (2).
	 
($\Longleftarrow$)	Assume that (1), (2) and (3) hold. By \cite[Theorem 4.7]{gr:saccfulpavpsgrt} we immediately obtain that $L_K(\mathcal{G})$ is simple. By Lemmas~\ref{localunit} and~\ref{pisimple}, it suffice to show that $L_K(\mathcal{G})$ is not a division ring, and that for every pair of nonzero elements $\alpha, \beta$ in $L_K(\mathcal{G})$ there exist elements $a, b$ in $L_K(\mathcal{G})$ such that $a \alpha b = \beta$. Condition (3) implies that there exists a cycle $c=e_1e_2\cdots e_n$ with $n\ge 1$. By condition (2), $c$ has an exit, that is, there exists  $f\in \mathcal{G}^1$ such that $s(f)\in r(e_i)$ and $f\neq e_{i+1}$ for some $1\le i \le n$ (where $e_{n+1}:= e_1$), or $r(e_j)$ contains a sink $v$ for some $j$. If the first case occurs, then we have that $s^*_f s_{e_{i+1}} = 0$. If the second one occurs, then we have that $p_vs_{e_j} = p_v p_{s(e_j)}s_{e_j} = 0\cdot s_{e_j} = 0$. This implies that $L_K(\mathcal{G})$ has zero divisors, and thus $L_K(\mathcal{G})$ is not a division ring.

We now apply the Reduction Theorem \cite[Theorem 3.2]{gon:ratrt19} to find $\overline{a}, \overline{b} \in L_K(\mathcal{G})$ such that either $\overline{a}\alpha \overline{b} = p_{_A}$ for some nonempty set $A\in\mathcal{G}^0$, or $\overline{a}\alpha \overline{b} = \sum^m_{i=1}k_is^i_{c}$, where $c$ is a cycle without exit. By condition (2), the later can not happen, and so we must obtain that $\overline{a}\alpha \overline{b} = p_{_A}$. Take any $w\in A$. We then have
\begin{center}
$p_w\overline{a}\alpha \overline{b} = p_w p_{_A} = p_w.$	
\end{center} 
By condition (3), $w$ connects to a cycle $c'=e'_1e'_2\cdots e'_n$, and so there exists path $p\in \mathcal{G}^*$ with $s(p)=w$ and $s(e'_j)\in r(p)$ for some $1\le j \le n$. We note that $s_p^*p_ws_p = s_p^*s_p = p_{r(p)}$. Let $v := s(e'_j)$, $a':=s_p^*$ and $b':=s_pp_v$. We have that $$a'p_wb' = s_p^*p_ws_pp_v=p_{r(p)}p_v=p_v.$$ 

Since $\mathcal{H}_{\mathcal{G}}= \{\varnothing, \mathcal{G}^0\}$ and by \cite[Lemma 6.1]{Larki:piapioua19}, there exist two distinct cycles $p, q $ with $v:=s(p)= s(q)$, $p$ is not a subpath of $q$, and $q$ is also not a subpath of $p$. For any $m\ge 1$ let $c_m$ denote the path $p^{m-1}q$, where $p^0:= p_v$. We then have that $s_{c_m}^*s_{c_n} = \delta_{m, n}p_{r(q)}$ for every $m,n \ge 1$, where $\delta$ is the Kronecker delta, and so $s_{c_m}^*s_{c_n}p_v = \delta_{m, n}p_{r(q)}p_v=\delta_{m, n}p_v$ for every $m,n\ge 1$ (since $v\in r(q)$).

Since $L_K(\mathcal{G})$ is simple and $p_v\neq 0$, there exist $\{a_i, b_i\in L_K(\mathcal{G})\mid 1\ge i\ge t\}$ such that
$\beta =\sum\limits_{i=1}^ta_ip_vb_i$. Let $a_{\beta} :=\sum\limits_{i=1}^ta_is_{c_i}^*$ and $b_{\beta} := \sum\limits_{j=1}^ts_{c_j}b_j$. We then obtain that $$a_{\beta}p_vb_{\beta}=(\sum\limits_{i=1}^ta_is_{c_i}^*)p_v(\sum\limits_{j=1}^ts_{c_j}b_j) = \sum\limits_{i=1}^ta_ip_vb_j = \beta.$$

Finally, letting $a=a_{\beta}a'p_w\overline{a}$ and  $b=\overline{b}b'b_{\beta}$, we immediately get that $a\alpha b = \beta$, as desired, thus finishing the proof.
\end{proof} 

%In \cite[Theorem 3.4]{ima:tlpaou}, the authors completely described graded ideals of ultragraph Leavitt path algebras.In \cite[Lemma 3.1]{ima:tlpaou}, the authors showed that if $\mathcal{H}$ is a saturated and hereditary subset of $\mathcal{G}^0$, then $I(\mathcal{H})$ is spanned as a $K$-vertor space by $\{s_{\alpha}p_{_A}s^*_{\beta}\mid \alpha, \beta \in \mathcal{G}^*, A\in\mathcal{H}, r(\alpha)\cap A\cap r(\beta) \neq \varnothing\}$. 
For an ultragraph $\mathcal{G}$ and subset $\mathcal{H}\subseteq \mathcal{G}$, we denote by $I(\mathcal{H})$ the ideal of $L_K(\mathcal{G})$ generated by the idempotents $\{p_{_A}\mid A \in \mathcal{H}\}$. If $\mathcal{H} = \{v\}$ with $v\in G^0$, then we denote $I(\mathcal{H})$ by $I(v)$.
Let $c= e_1e_2\cdots e_n$ be a cycle without exits in $\mathcal{G}$. We then have that $s(e_i)\neq s(e_j)$ for all $1\le i\neq j\le n$, and $r(e_i) = \{s(e_{i+1})\}$ for all $1\le i\le n$ (where $e_{n+1}:= e_1$). We denote by $c^0$ the set of all subsets of $\{s(e_i)\mid 1\le i\le n\}$. It is obvious that $c^0$ is a hereditary subset of $\mathcal{G}^0$. We say that a path $p$ \textit{ends at a vertex} $v$ if $r(p) =\{v\}$.
The following lemma provides us with a complete description of the ideal $I(c^0)$.  

\begin{lem}\label{graphonecyclewithoutexit}
Let $\mathcal{G}$ be an ultragraph, $K$ a field, and $c$ a cycle without exits. Let $v:=s(c)$, and let $\Lambda$ be the (possibly infinite) set of all finite paths in $\mathcal{G}$ which end at $v$, but which do not contain all the edges of $c$. Then $$I(\overline{c^0})=I(c^0)=I(v)\cong M_{\Lambda}(K[x,x^{-1}]),$$
where $\overline{c^0}$ is the smallest saturated hereditary subset of $\mathcal{G}^0$ containing $c^0$.
\end{lem}
\begin{proof}
We first claim that $$I(v)=\text{Span}_K\{s_{\alpha}s_{\beta}^*\mid \alpha, \beta\in\mathcal{G}^*, r(\alpha) = \{v\}=r(\beta)\}.$$ 
Indeed, we denote by $J$ the right-hand side of the above equality. We note that for every $\alpha,\beta, \mu, \nu \in \mathcal{G}^*\setminus \mathcal{G}^0$ and every $A,B \in \mathcal{G}^0$, we have  \begin{align*}(s_{\alpha}p_{_A}s_{\beta}^*)(s_{\mu}p_{_B}s_{\nu}^*)= \begin{cases} s_{\alpha\mu'}p_{_B}s_{\nu}^* &\textnormal{if }  \mu=\beta\mu', s(\mu')\in A\cap r(\alpha) \textnormal{ and } |\mu'|\ge 1, \\ s_{\alpha}p_{_{A\cap r(\beta)\cap B}}s_{\nu}^* & \textnormal{if } \mu=\beta, \\ s_{\alpha}p_{_A}s_{\nu\beta'}^* &\textnormal{if }  \beta=\mu\beta', s(\beta')\in B\cap r(\nu) \textnormal{ and } |\beta'|\ge 1, \\ 0 &\textnormal{otherwise}.\end{cases}\end{align*} 
By this note and Lemma~\ref{graded} (2), we immediately get that $J$ is an ideal of $L_K(G)$. Moreover, we have $p_v = p_vp_v\in J$, i.e. $J$ is an ideal of $L_K(\mathcal{G})$ containing $p_v$, and so $I(v)\subseteq J$. For $\alpha, \beta\in\mathcal{G}^*$ with $r(\alpha) = \{v\}=r(\beta)$, we obtain that $s_{\alpha}s_{\beta}^* = s_{\alpha}p_vs_{\beta}^*\in I(v)$, and hence $J\subseteq I(v)$, thus proving the claim. 

It is obvious that $I(v)\subseteq I(c^0) \subseteq I(\overline{c^0})$. To show the inverse inclusions, we write $c= e_1e_2\cdots e_n$ $(n\ge 1)$ and $v_i:= s(e_{i+1})$ for all $1\le i\le n-1$ (note that $v= s(c) = s(e_1)$). For $1\le i\le n-1$, there exists a path $p:= e_1\cdots e_{i}$ with $s(p) = v$ and $r(p) = \{v_i\}$, and so $p_{v_i} = s^*_p s_p = s^*_pp_v s_p\in I(v)$. Take any a non-empty set $A\in c^0$. We then have that $A$ is a subset of $\{v, v_i\mid 1\le i\le n-1\}$, and $p_{_A} = \sum_{w\in A}p_w \in I(v)$. This implies that $I(c^0)\subseteq I(v)$, and so $I(v)= I(c^0)$.

In order to see that $I(c^0)=I(\overline{c^0})$, it is sufficient to prove that
\begin{center}
$p_{_B} \in I(c^0)$\quad\quad for all $B\in \overline{c^0}$. 	
\end{center} 
We first note that since $c^0$ is a hereditary subset of $\mathcal{G}^0$ and by Lemma~\ref{hereditarysaturationsmallest}, we have that $$\overline{c^0} = \bigcup\limits_{i=0}^{\infty}\mathcal{H}_i,$$ 
where $\mathcal{H}_0:= c^0$ and  
\begin{center} 
	$\mathcal{H}_{n+1}:= \{A\cup F\mid A\in \mathcal{H}_n$ and $F \text{ is a finite subset of } S_n\}$,
\end{center}
\begin{center} 
$S_n:= \{w\in G^0\mid 0<|s^{-1}(w)|<\infty$ and $\{r(e)\mid s(e)=w\}\subseteq \mathcal{H}_n\rbrace$. 
\end{center}  

We shall claim by proving inductively that $p_{_B}\in I(c^0)$ for all $B\in \mathcal{H}_n$. Indeed, if $n=0$, then $B\in c^0$, and so $p_{_B}\in I(c^0)$. Now we proceed inductively. For $n>1$, we have that $B = A\cup F$ for some $A\in \mathcal{H}_{n-1}$ and for some finite subset $F\subseteq S_{n-1}$. As is shown in the proof of \cite[Lemma 3.12]{tomf:soua} that $\mathcal{H}_i$ is hereditary for all $i\in \mathbb{N}$. By this fact and the induction hypothesis, we get that $p_{_A}$, $p_{_{(A\cap F)}}\in I(c^0)$, and $p_{r(e)}\in I(c^0)$ for all $e\in \mathcal{G}^1$ with $r(e)\in \mathcal{H}_{n-1}$.

Write $F=\{w_1, w_2,\hdots, w_k\}\subseteq S_{n-1}$. Then, for $1\le i\le k$, $w_i$ is a regular vertex and $r(e)\in \mathcal{H}_{n-1}$ for all $e\in s^{-1}(w_i)$. This implies that

$$p_{w_i}=\sum\limits_{e\in s^{-1}(w_i)}s_es_e^*=\sum\limits_{e\in s^{-1}(w_i)}s_ep_{r(e)}s_e^*\in I(c^0),$$ 
so $p_{_F}=\sum\limits_{i=1}^kp_{w_i}\in I(c^0)$ and $p_{_B} = p_{_{A\cup F}} = p_{_A} + p_{_F} - p_{_{A\cap F}}\in I(c^0)$, showing the claim. Using this claim and the above note, we immediately get that $p_{_B} \in I(c^0)$ for all $B\in \overline{c^0}$, and so $I(c^0)=I(\overline{c^0})$.

Consider the family $$\mathcal{B}:=\{s_{\alpha}s_c^ks_{\beta}^*\mid \alpha, \beta \in \Lambda, k\in \mathbb{Z}\},$$ where as usual $s_c^0$ denotes $p_v$ and $c^k$ denotes $(s_c^*)^{-k}$ for $k<0$. 

We note that since $c$ is a cycle without exits, we have $s_cs^*_c = p_v$. Also, for $\alpha\in \mathcal{G}^*$ with $r(\alpha) = \{v\}$, it may be written in the form: $\alpha = p c^m$ for some $p\in \Lambda$ and for some $m\in \mathbb{N}$ (where $c^0:= v$). Using this note, we immediately obtain that for all $p, q\in\mathcal{G}^*$ with $r(p) = \{v\}=r(q)$, $s_ps^*_q = s_{\alpha}s_c^ks_{\beta}^*$ for some $\alpha, \beta\in \Lambda$ and for some $k\in \mathbb{Z}$. This implies that $\mathcal{B}$ generates $I(v)$.

We next claim that $\mathcal{B}$ is a $K$-linearly indenpendent set. Consider the equation
\begin{equation*}
\sum\limits_{i=1}^nk_is_{\alpha_i}s_c^{m_i}s^*_{\beta_i}=0\tag{\mbox{$\ast$}}
\end{equation*}
where $k_i, m_i\in \mathbb{N}$ and $\alpha_i, \beta_i\in \Lambda$. By induction on $n$ we prove that $k_i=0$ for all $1\le i\le n$. If $n = 1$, then we have $k_1s_{\alpha_1}s_c^{m_1}s^*_{\beta_1}=0$, and so $k_1p_{v} = k_1(s^*_c)^{m_i}s^*_{\alpha_1} s_{\alpha_1}s_c^{m_1}s^*_{\beta_1} s_{\beta_1} = 0$. By Lemma~\ref{graded} (1), $p_v\neq 0$, and hence $k_1 = 0$. Now we proceed inductively. For $n>1$, if $s(\alpha_1) \neq s(\alpha_j)$ for some $1 \le j\le n$, then we have 
\begin{equation*}
\sum\limits_{i=2}^nk_ip_{s(\alpha_j)}s_{\alpha_i}s_c^{m_i}s^*_{\beta_i} = p_{s(\alpha_j)}(\sum\limits_{i=1}^nk_is_{\alpha_i}s_c^{m_i}s^*_{\beta_i})=0.\tag{\mbox{$\ast\ast$}}
\end{equation*}
Using equations  $(\ast)$ and $(\ast\ast)$, and the induction hypothesis, we immediately get that $k_i = 0$ for all $i$. If $s(\beta_1) \neq s(\beta_j)$ for some $1 \le j\le n$, then we have
\begin{equation*}
\sum\limits_{i=2}^nk_is_{\alpha_i}s_c^{m_i}s^*_{\beta_i}p_{s(\beta_j)} = (\sum\limits_{i=1}^nk_is_{\alpha_i}s_c^{m_i}s^*_{\beta_i})p_{s(\beta_j)}=0.\tag{\mbox{$\ast\ast\ast$}}
\end{equation*}
Using equations  $(\ast)$ and $(\ast\ast\ast)$, and the induction hypothesis, we obtain that $k_i = 0$ for all $i$.

Consider the case that $s(\alpha_1) = s(\alpha_i)$  and $s(\beta_1) = s(\beta_i)$ for all $i$. Then, since $r(\alpha_i) = \{v\} = r(\beta_i)$ for all $i$, we have that
\begin{center}
$s^*_{\alpha_1}s_{\alpha_i} =\delta_{\alpha_1, \alpha_i}p_{r(\alpha_1)} =\delta_{\alpha_1, \alpha_i}p_{v}$
and $s^*_{\beta_i}s_{\beta_1} = \delta_{\beta_i, \beta_1}p_{r(\beta_1)}=\delta_{\beta_i, \beta_1}p_{v}$	
\end{center}  
for all $i$ (where $\delta$ is the Kronecker delta), and 
\begin{equation*}
0 = s^*_{\alpha_1}(\sum\limits_{i=1}^nk_is_{\alpha_i}s_c^{m_i}s^*_{\beta_i})s_{\beta_1}= \sum_{j=1}^r k_{i_j}s^{m_{i_j}}_c
\end{equation*}
where $1\le r\leq n$, $\{i_j\mid 1\le j\le r\}\subseteq \{1, 2,\hdots, n\}$ and $\alpha_1 = \alpha_{i_j}$, $\beta_1 = \beta_{i_j}$ for all $1\le j\le r$, and $m_{i_j}$'s are distinct.
Now the grading in $L_K(\mathcal{G})$ (see Lemma~\ref{graded} (2)) shows that $k_{i_j} = 0$ for all $1\le j\le r$. From this observation and equation $(\ast)$, and by the induction hypothesis, we obtain that $k_i =0$ for all $i$, showing the claim. Therefore $\mathcal{B}$ is a $K$-basis for $I(v)$.

We define $\phi : I(v)\longrightarrow M_{\Lambda}(K[x,x^{-1}])$ by setting $\phi (s_{\alpha}s_c^ks_{\beta}^*)=x^kE_{\alpha,\beta}$ for each $s_{\alpha}s_c^ks_{\beta}^*\in \mathcal{B}$, where $x^kE_{\alpha,\beta}$ denotes the element of $M_{\Lambda}(K[x,x^{-1}])$ which is $x^k$ in the $(\alpha, \beta)$ entry, and zero otherwise. Then we easily check that $\phi$ is a $K$-algebra isomorphism, thus finishing the proof.	
\end{proof}

We now have all the tools necessary to generalize \cite[Theorem~3.1.14]{AAS} which the authors of \cite{AAS} call the Trichotomy Principle for graded simple Leavitt path algebras of graphs. We prove this principle for graded simple ultragraph Leavitt path algebras.

\begin{thm}\label{TrichotomyPrinciple}
Let $\mathcal G$ be an ultragraph and $K$ a field. If $L_K(\mathcal G)$ is graded simple, then exactly one of the following occurs:
	
$(1)$ $L_K(\mathcal{G})$ is locally matricial, or
	
$(2)$  $L_K(\mathcal{G}) \cong M_{\Lambda}(K[x,x^{-1}])$ for some set $\Lambda$, or
	
$(3)$  $L_K(\mathcal{G})$ is purely infinite simple.
\end{thm} 
\begin{proof}
By \cite[Theorem~3.4]{ima:tlpaou}, the graded simplicity of $L_K(\mathcal{G})$ is equivalent to that $\mathcal{H}_{\mathcal{G}}= \{\varnothing, \mathcal{G}^0\}$. The three possibilities given in the statement correspond precisely to whether: (1) $\mathcal{G}$ contains no cycles; resp., (2) contains exactly one cycle; resp., (3) contains at least two cycles.

If $\mathcal{G}$ contains no cycle then (1) follows from Theorem~\ref{regularity}. Consider the case that $\mathcal{G}$ contains least two cycles. Let $c_1 =e_1\cdots e_n$ and $c_2=f_1\cdots f_m$ be two distinct cycles in $\mathcal{G}$. We then have two infinite paths $(c_1)^{\infty} = c_1c_1\cdots c_1\cdots$ and $(c_2)^{\infty} = c_2c_2\cdots c_2\cdots$.
Applying Lemma~\ref{gradedsimpleulgraph} (1), we immediately get that every vertex $v\in G^0$ connects to both $(c_1)^{\infty}$ and $(c_2)^{\infty}$; equivalently, $v$ connects to both
$c_1$ and $c_2$. Consequently, every cycle in $\mathcal{G}$ has an exit. Then, by Theorem~\ref{purelysimple}, we have that $L_K(\mathcal{G})$ is purely infinite simple.
	
Now suppose that $\mathcal{G}$ contains exactly one cycle $c=\alpha_1\alpha_2\cdots\alpha_n$. If $c$ has exits then  there exists  $f\in \mathcal{G}^1$ such that $s(f)\in r(e_i)$ and $f\neq e_{i+1}$ for some $1\le i \le n$ (where $e_{n+1}:= e_1$), or $r(e_j)$ contains a sink $v$ for some $j$. If the first case  occurs, then by Lemma~\ref{gradedsimpleulgraph} (1), every vertex in $r(f)$ connects to $c$, and so $\mathcal{G}$ has at least two cycles, a contradiction. If the second one occurs, then by Lemma~\ref{gradedsimpleulgraph} (1), $v$ connects to the infinite path $c^{\infty}:= c\cdots c\cdots$, and so $v$ is not a sink, a contradiction. Therefore, $c$ is a cycle without exits. Now, by Lemma~\ref{graphonecyclewithoutexit}, $I(\overline{c^0})\cong M_{\Lambda}(K[x, x^{-1}])$, where $\Lambda$ is the set of all finite paths in $\mathcal{G}$ which end at $v$, but which do not contain all the edges of $c$.
Since $\overline{c^0}$ is a non-empty saturated hereditary subset of $\mathcal{G}^0$ and $\mathcal{H}_{\mathcal{G}}= \{\varnothing, \mathcal{G}^0\}$, we must have that $\overline{c^0} = \mathcal{G}^0$, and so $I(\overline{c^0}) = L_K(\mathcal{G})$ by Lemma~\ref{localunit}. This implies that $L_K(\mathcal{G})\cong M_{\Lambda}(K[x,x^{-1}])$, thus finishing our proof.

 \end{proof} 

\vskip 0.5 cm \vskip 0.5cm {

\end{document}